\documentclass[11pt,letterpaper]{amsart}
\usepackage[colorlinks=true,linkcolor=blue,citecolor=blue]{hyperref}
\usepackage{hyperref}

\numberwithin{equation}{section}
\newtheorem{theorem}{Theorem}[section]
\newtheorem{lemma}[theorem]{Lemma}
\newtheorem{corollary}[theorem]{Corollary}
\newtheorem{proposition}[theorem]{Proposition}
\newtheorem{remark}[theorem]{Remark}

\begin{document}
\title[A Christ-Kiselev theorem in function lattices]{A Christ-Kiselev maximal theorem in quasi-Banach function lattices}

\author{Mieczys{\l}aw Masty{\l}o}
\address{Faculty of Mathematics and Computer Science,
Adam Mickiewicz University, Pozna{\'n},
Uniwersytetu Pozna{\'n}skiego 4,
61-614 Pozna{\'n}, 
Poland}
\email{mieczyslaw.mastylo@amu.edu.pl}
\thanks{The first author was supported by the National Science Centre, Poland, Project
2019/33/B/ST1/00165}

\author{Gord Sinnamon}
\address{Department of Mathematics, Western University, London, Canada}
\email{sinnamon@uwo.ca}
\thanks{The second author was supported by the Natural Sciences and Engineering Research Council of Canada}

\begin{abstract}
\noindent
A Christ-Kiselev maximal theorem is proved for linear operators between quasi-Banach function lattices satisfying certain lattice geometrical conditions. The result is further explored for weighted Lorentz spaces, classical Lorentz spaces, and Wiener amalgams of Lebesgue function and sequence spaces. Extensions are made to K\"othe dual operators and to operators on interpolation spaces of quasi-Banach function lattices. Several applications to maximal Fourier operators are presented. 
\end{abstract}
\maketitle

%\setcounter{tocdepth}{1}
%\tableofcontents

\section{Introduction} The maximal theorem given by Christ and Kiselev in \cite{ChK1} starts with a simple, flexible definition of a maximal operator based on a given linear operator and a given collection of sets. It shows that the maximal operator is bounded between two $L^p$ spaces whenever the original operator is. The ease and generality of the approach has simplified proofs and provided new insights in a wide variety of applications. 

\begin{theorem}[Christ-Kiselev Maximal Theorem]\label{CKorig} Let $(\Omega, \mu)$ and $(\widetilde\Omega, \nu)$ be measure spaces and let $\mathcal A$ be a countable, totally ordered collection of measurable subsets of $\Omega$. If $1\le p<q\le\infty$, and $T$ is a bounded, linear operator from $L^p_\mu$ to $L^q_\nu$ then the sublinear operator $T^*$, defined by
\[
T^{*}f(x) = \sup_{A\in\mathcal A} |T(f\chi_A)(x)|,
\]
is also bounded from  $L^p_\mu$ to $L^q_\nu$ and 
$
\|T^{*}\| \le \big(1- 2^{1/q-1/p}\big)^{-1}\|T\|.
$
\end{theorem}

Our objective is to widen the applicability of this result by replacing the $L^p$ spaces by quasi-Banach function lattices. The hypothesis $p<q$ above is essential, so in any extension of Theorem \ref{CKorig} this separation between the domain and range spaces must be preserved. We employ upper and lower $p$-estimates, weak forms of convexity and concavity in function lattices, to accomplish this. Some useful tools involving these estimates are presented in the next section.

The general extension is established next and then we turn our attention to special cases, specifically, to the class of Lorentz spaces $\Lambda_{r,w}$ and to a class of Weiner amalgam spaces. The former contains the classical Lorentz spaces $L_{p,r}$, for which an extension is already known, see \cite{CKL}. The latter is a natural choice for an extension, in fact, the lack of Christ-Kiselev theorem in Weiner amalgam spaces was pointed out in \cite[Remark 1.5]{KKS}.

Duality in function lattices, i.e., the K\"othe dual, produces nothing new in the scale of $L^p$ spaces, but can be a rich source of useful lattices in the larger context. We explore the connection between the maximal operator based on a given linear operator and the maximal operator based on its dual operator. Interpolation between function lattices is an even richer source of lattices and of operators on them. We show that the boundedness of a maximal operator between interpolation spaces can be guaranteed by conditions on the endpoint spaces. 

%Along the way, applications related to the Fourier transform are used to illustrate the methods. See Corollary \ref{FTLor}, Theorems \ref{SS} and \ref{MPZ}, Corollary \ref{Gam} and Proposition \ref{BaSe}.

\subsection{Preliminaries}

We use standard notation from Banach space theory. A quasi-norm on a real or complex vector space $X$ satisfies the properties of a norm, but with the triangle inequality replaced by  $\|x+y\|\le \kappa(\|x\|+\|y\|)$ for all $x,y\in X$, for some $\kappa\ge1$. A complete, quasi-normed space is called a quasi-Banach space.

For background on Banach lattices and Banach function lattices we refer the reader to \cite{LT}. For a complete, $\sigma$-finite measure space $(\Omega, \mu)=(\Omega, \Sigma,\mu)$, $L^0_\mu $ denotes the space of all (equivalence classes of) real- or complex-valued $\Sigma$-measurable functions on $\Omega$ that are finite $\mu$-a.e. A quasi-Banach space $E\subset L^0_\mu $ is said to be a quasi-Banach function lattice over $\Omega$ if for all $f, g\in L^0_\mu $ with $|g| \le |f|$ $\mu$-a.e.\ and
$f\in E$, we have $g\in E$ and $\|g\|_E \le \|f\|_E$. 

If $E$ is a Banach function lattice, $E'$ denotes the K\"othe dual of $E$, defined to be the vector space of all $g\in L^0_\mu $ such that
\[
\|g\|_{E'}=\sup\Big\{\int_\Omega|fg|\,d\mu:f\in E,\|f\|_E\le1\Big\}<\infty.
\] 
Since $\mu$ is $\sigma$-finite, $\|g\|_{E'}$ is unchanged if the supremum is restricted to the integrable functions in $f\in E$.

The \emph{distribution function}, $\mu_f$, and \emph{nonincreasing rearrangement}, $f^*$, of an $f\in L^0_\mu $, are defined by $\mu_f(\tau)=\mu(\{x\in\Omega: |f(x)|> \tau\})$ for $\tau \ge 0$, and $f^{*}(t)=\inf \{\tau \ge 0: \mu_f(\tau)\le t\}$ for $t\ge 0$. A quasi-Banach function lattice $E$ over $\Omega$ is called \emph{rearrangement-invariant} if $\|f\|_E = \|g\|_E$ for all $f, g\in E$ that satisfy $\mu_f = \mu_g$.

The Fourier transform of a Lebesgue integrable function $f:\mathbb R^n\to\mathbb C$ is defined by 
$$
\mathcal Ff(x)=\hat f(x)=\int_{\mathbb R^n}e^{-2\pi ix\cdot t}f(t)\,dt,
$$
suitably modified to define $\hat\mu$ for a Borel probability measure $\mu$ on $\mathbb R^n$. Extensions of the Fourier transform for $f$ in various spaces are also denoted $\mathcal Ff$ and $\hat f$.

\section{Upper and lower $p$-estimates in function lattices}
Let $X$ be a quasi-Banach lattice and $0< p \le \infty$. For each $n \in \mathbb{N}$ the $p$-convexity constant for $n$-vectors, $M^{(p),n}(X)$, is the least $C>0$ such that for all $x_1, \dots, x_n \in X$,
\[
\Big\|\Big(\sum_{j=1}^n |x_j|^p\Big)^{1/p}\Big\|_X \le C\Big(\sum_{j=1}^n \|x_j\|_{X}^p\Big)^{1/p}.
\]
For each $n \in \mathbb{N}$ the $p$-concavity constant for $n$-vectors, $M_{(p),n}(X)$, is the least $C>0$ such that for all $x_1, \dots, x_n \in X$,
\[
\Big(\sum_{j=1}^n \|x_j\|_{X}^p\Big)^{1/p} \le C\Big\|\Big(\sum_{j=1}^n |x_j|^p\Big)^{1/p}\Big\|_X.
\]
We make the usual modification for $p=\infty$ in the above expressions.
The space $X$ is called \emph{$p$-convex} if $M^{(p)}(X) = \sup_{n\geq 1} M^{(p),n}(X)<\infty$. It is called \emph{$p$-concave} if $M_{(p)}(X) = \sup_{n\geq 1} M_{(p), n}(X) <\infty$. 
 
For each $n\in \mathbb N$, $u^{(p),n}(X)$ is the least $C$ such that for any disjoint $x_1, \dots, x_n\in X$,
\[
\Big\|\sum_{j=1}^n x_j\Big\|_X \le C \Big(\sum_{j=1}^n \|x_j\|_{X}^p\Big)^{1/p}
\]
and $\ell_{(p),n}(X)$ is the least $C$ such that for any disjoint $x_1, \dots, x_n\in X$,
\[
\Big(\sum_{j=1}^n \|x_j\|_{X}^p\Big)^{1/p} \le C \Big\|\sum_{j=1}^n x_j \Big\|_X.
\]
Once again, we modify these expressions when $p=\infty$.
We say $X$ satisfies an \emph{upper $p$-estimate} if $u^{(p)}(X) =\sup_{n\geq 1}u^{(p), n}(X)< \infty$, and satisfies a \emph{lower $p$-estimate} if $\ell_{(p)}(X) = \sup_{n\ge 1} \ell_{(p),n}(X) < \infty$.
It is easy to see that if $X$ is $p$-convex, then it satisfies an upper $p$-estimate and if $X$ is $p$-concave, then it satisfies a lower $p$-estimate. 

The next four results are based on standard facts about upper and lower $p$-estimates but have been adapted to suit our situation. Let $E$ be a quasi-Banach function lattice. We say $E$ has the {\it Fatou property} if the following condition holds for all measurable $f$: If $0\le f_j\in E$ for $j\in\mathbb N$ and $f_j$ increases to $f$ pointwise a.e., then $f\in E$ and $\|f_j\|_E \to \|f\|_E$. We say $E$ has the {\it weak Fatou property} if the condition holds for all $f\in E$. A well-known theorem due to Nakano (see \cite{KA}), shows that if $E$ has the weak Fatou property, then
\[
\|f\|_E = \sup_{\|g\|_{E'} \leq 1} \bigg|\int_{\Omega} f g\,d\mu\bigg|.
\]

It is proved in \cite[Proposition 1.f.5]{LT} that if $1<p<\infty$ and $1/p + 1/p' =1$, then a Banach lattice $X$ satisfies an upper $p$-estimate
if and only if $X^{*}$ satisfies a lower $p'$-estimate and $X$ satisfies a lower $p$-estimate
if and only if $X^{*}$ satisfies an upper $p'$-estimate. We will need the following quantitative version for a Banach function lattice and its K\"othe dual space.
\begin{proposition} \label{ulndual}
Let $E$ be a Banach function lattice over $(\Omega,\mu)$ with the weak Fatou property. If $1\le p\le\infty$ and $n\in \mathbb{N}$, then $\ell_{(p), n}(E) = u^{(p'),n}(E')$ and $u^{(p), n}(E) = \ell_{(p'),n}(E')$. Also $\ell_{(p)}(E) = u^{(p')}(E')$ and $u^{(p)}(E) = \ell_{(p')}(E')$, whether finite or infinite.
\end{proposition}
\begin{proof} Let $g_1, \dots, g_n \in E'$ have disjoint supports $A_1,\dots,A_n$, respectively. If $f\in E$, with $\|f\|_E\le1$, then $\|f\chi_{A_1}\|_E^p+\dots+\|f\chi_{A_n}\|_E^p\le \ell_{(p), n}(E)^p$ if $p<\infty$ and $\max\{\|f\chi_{A_1}\|_E,\dots,\|f\chi_{A_n}\|_E\}\le \ell_{(p), n}(E)$ if $p=\infty$. Therefore,
\[
\int_\Omega\Big|f\sum_{j=1}^ng_j\Big|\,d\mu
=\sum_{j=1}^n \int_{\Omega} |f\chi_{A_j} g_j|\,d\mu
\le \sum_{j=1}^n \|f\chi_{A_j}\|_E \|g_j\|_{E'}.
\]
H\"older's inequality shows that the last expression is no greater than 
\[
\ell_{(p), n}(E)\Big(\sum_{j=1}^n \|g_j\|_{E'}^{p'}\Big)^{1/p'},
\]
with the usual modification when $p'=\infty$. Take the supremum over all such $f$ to see that $u^{(p'), n}(E') \le \ell_{(p), n}(E)$.

To establish the reverse inequality, suppose $f_1, \dots, f_n\in E$ have disjoint supports $A_1,\dots,A_n$, respectively. Fix nonnegative $a_1,\dots, a_n$ that satisfy $a_1^{p'}+\dots+a_n^{p'}\le1$ when $p'<\infty$ and $\max\{a_1,\dots,a_n\}\le1$ when $p'=\infty$. Let $g_1,\dots,g_n\in E'$ with $g_j$ supported on $A_j$ and $\|g_j\|_{E'}\le1$ for each $j$. Then
\[
\Big\|\sum_{j=1}^na_jg_j\Big\|_{E'}
\le u^{(p'), n}(E')\Big(\sum_{j=1}^na_j^{p'}\|g_j\|^{p'}_{E'}\Big)^{1/p'}\le u^{(p'), n}(E').
\]
Therefore,
\[
\sum_{j=1}^na_j\Big|\int_{A_j}f_jg_j\,d\mu\Big|
\le\int_{\Omega}\Big|\sum_{j=1}^nf_j\Big|\Big|\sum_{j=1}^na_jg_j\Big|\,d\mu
\le u^{(p'), n}(E')\Big\|\sum_{j=1}^n f_j\Big\|_E.
\]
Take the supremum over all such $g_j$ for each $j$ and apply Nakano's Theorem to get 
\[
\sum_{j=1}^na_j\|f_j\|_E \le u^{(p'), n}(E')\Big\|\sum_{j=1}^n f_j\Big\|_E.
\]
Then take the supremum over all such $a_1,\dots, a_n$, to see that $\ell_{(p), n}(E)\le u^{(p'), n}(E')$.

The equation $u^{(p), n}(E) = \ell_{(p'),n}(E')$ is proved similarly. The last two equations in the proposition follow from the first two by taking the supremum over $n$.
\end{proof}

The next lemma is closely related to \cite[Theorem 1.f.11]{LT}.
\begin{lemma} \label{ellu1E}
Let $0<p\le\infty$ and let $E$ be a quasi-Banach function lattice satisfying a lower $p$-estimate. Then there is a lattice quasi-norm $\|\cdot\|_{E_{(p)}}$ on $E$ such that 
\begin{itemize}
\item[{\rm(i)}] $\ell_{(p)}(E_{(p)})=1$, where $E_{(p)}$ denotes $E$ with the quasi-norm $\|\cdot\|_{E_{(p)}}${\rm;} 
\item[{\rm(ii)}] for all $f\in E$, $\|f\|_E\le\|f\|_{E_{(p)}}\le\ell_{(p)}(E)\|f\|_E${\rm;}  and
\item[{\rm(iii)}]  if $E$ is a normed space and $p\ge1$, then $\|\cdot\|_{E_{(p)}}$ is a norm.
\end{itemize}

\end{lemma}
\begin{proof} We prove only the case $p<\infty$. The same proof, with appropriate modifications to some of the expressions, works when $p=\infty$.  Choose $\kappa$ so that $\|f+g\|_E\le \kappa(\|f\|_E+\|g\|_E)$ for all $f, g\in E$, taking $\kappa=1$ if $E$ is a normed space. For each $f\in E$, set
\[
\|f\|_{E_{(p)}}=\sup\Big(\sum_{j=1}^n\|f_j\|_E^p\Big)^{1/p},
\]
where the supremum is taken over all finite collections $f_1,\dots,f_n$ of disjointly supported functions that sum to $f$. Taking $n=1$ and $f_1=f$ shows that $\|f\|_E\le\|f\|_{E_{(p)}}$ and the definition of $\ell_{(p)}(E)$ shows $\|f\|_{E_{(p)}}\le\ell_{(p)}(E)\|f\|_E$. 

The first of these two inequalities ensures that for each $f\in E$, $\|f\|_{E_{(p)}}\ge0$ with equality only if $f=0$ $\mu$-a.e. For $\alpha\in\mathbb R$, it is routine to verify that $\|\alpha f\|_{E_{(p)}}=|\alpha|\|f\|_{E_{(p)}}$.  To prove that $\|\cdot\|_{E_{(p)}}$ has the lattice property, suppose $h$ is a measurable function, $f\in E$ and $|h|\le|f|$. Since $E$ is a lattice, $h\in E$. Take $h_1,\dots,h_n$ to be disjointly supported functions in $E$ that sum to $h$ and take $H_1,\dots,H_n$ to be a measurable partition of the underlying measure space such that $h_j$ is supported on $H_j$ for all $j$. Then $f\chi_{H_1}+\dots+f\chi_{H_n}=f$ and $|h_j|\le |f\chi_{H_j}|$ for all $j$. We have,
\[
\Big(\sum_{j=1}^n\|h_j\|_E^p\Big)^{1/p}\le\Big(\sum_{j=1}^n\|f_j\chi_{H_j}\|_E^p\Big)^{1/p}\le\|f\|_{E_{(p)}}.
\]
Taking the supremum over all such $h_1,\dots,h_n$ shows that $\|h\|_{E_{(p)}}\le\|f\|_{E_{(p)}}$.

For the triangle inequality, let $f,g\in E$ and suppose $h_1,\dots,h_n$ are disjointly supported and sum to $f+g$. Take $H_1,\dots,H_n$ as above, and observe that $h_j=f\chi_{H_j}+g\chi_{H_j}$ for each $j$, $f\chi_{H_1}+\dots+f\chi_{H_n}=f$ and $g\chi_{H_1}+\dots+g\chi_{H_n}=g$. Then
\begin{align*}
\Big(\sum_{j=1}^n&\|h_j\|_E^p\Big)^{1/p}=\Big(\sum_{j=1}^n\|f\chi_{H_j}+g\chi_{H_j}\|_E^p\Big)^{1/p}\\
&\le \kappa\Big(\sum_{j=1}^n(\|f\chi_{H_j}\|_E+\|g\chi_{H_j}\|_E)^p\Big)^{1/p}\\
&\le \kappa\max\{1,2^{1/p-1}\}\Big(\Big(\sum_{j=1}^n\|f\chi_{H_j}\|_E^p\Big)^{1/p}+\Big(\sum_{j=1}^n\|g\chi_{H_j}\|_E^p\Big)^{1/p}\Big)\\
&\le \kappa\max\{1,2^{1/p-1}\}(\|f\|_{E_{(p)}}^p+\|g\|_{E_{(p)}}^p).
\end{align*}
Taking the supremum over all such $h_1,\dots,h_n$ shows that
\[
\|f+g\|_{E_{(p)}}\le \kappa\max\{1,2^{1/p-1}\}(\|f\|_{E_{(p)}}+\|g\|_{E_{(p)}}).
\]
This proves that $\|\cdot\|_{E_{(p)}}$ is a quasi-norm, which is a norm if $E$ is a normed space and $p\ge1$.

It remains to show that $\ell_{(p)}(E_{(p)})=1$. Suppose $f_1,\dots,f_n$ are disjointly supported and sum to $f$ in $E$. Fix $\theta\in(0,1)$ and for each $j$, chose disjointly supported functions $f_{j,1},\dots, f_{j,n_j}$ that sum to $f_j$ and satisfy
\[
\theta\|f_j\|_{E_{(p)}}\le\Big(\sum_{k=1}^{n_j}\|f_{j,k}\|_E^p\Big)^{1/p}.
\]
For each $j$, the supports of $f_{j,1},\dots, f_{j,n_j}$ are contained in the support of $f_j$ so $\{f_{j,k}:1\le k\le n_j, 1\le j\le n\}$ is a collection of disjointly supported functions that sum to $f$. Therefore,
\[
\theta\Big(\sum_{j=1}^n\|f_j\|_{E_{(p)}}^p\Big)^{1/p}
\le\Big(\sum_{j=1}^n\sum_{k=1}^{n_j}\|f_{j,k}\|_E^p\Big)^{1/p}\le\|f\|_{E_{(p)}}.
\]
Let $\theta\to1$ to see that $\ell_{(p)}(E_{(p)})\le1$. But $\ell_{(p)}(E_{(p)})\ge1$ is trivial so this completes the proof.
\end{proof}
Combining the last two lemmas gives a result for upper $p$-estimates.
\begin{lemma}\label{comb} Let $1\le q\le\infty$ and let $F$ be a Banach function lattice with the Fatou property that satisfies an upper $q$-estimate. Then there is a lattice norm $\|\cdot\|_{F^{(q)}}$ on $F$ such that
\begin{itemize}
\item[{\rm(i)}] $u^{(q)}(F^{(q)})=1$, where $F^{(q)}$ denotes $F$ with the norm $\|\cdot\|_{F^{(q)}}${\rm;} and
\item[{\rm(ii)}] for all $g\in F$, $\|g\|_{F^{(q)}}\le\|g\|_F\le u^{(q)}(F)\|g\|_{F^{(q)}}$.
\end{itemize} 
\end{lemma}
\begin{proof} Let $p=q'$ and $E=F'$. Then $E$ is a normed function lattice and Proposition \ref{ulndual} shows that $E$ satisfies a lower $p$-estimate and $\ell_{(p)}(E)=u^{(q)}(F)$. Let $E_{(p)}$ be the space $E$, equipped with the norm $\|\cdot\|_{E_{(p)}}$ from Lemma \ref{ellu1E}. For each $f\in E$,
\[
\|f\|_{F'}\le\|f\|_{E_{(p)}}\le u^{(q)}(F)\|f\|_{F'}. 
\]
Let $F^{(q)}=(E_{(p)})'$ equipped with the dual norm.  Using Proposition \ref{ulndual} we get $u^{(q)}(F^{(q)})=\ell_{(p)}(E_{(p)})=1$. Since $F$ has the Fatou property, $F=E'$ so for each $g\in F$,
\[
\|g\|_{F^{(q)}}=\sup_{0\ne f\in E}\frac{\big|\int_\Omega fg\,d\mu\big|}{\|f\|_{E_{(p)}}}
\quad\text{and}\quad
\|g\|_F=\sup_{0\ne f\in E}\frac{\big|\int_\Omega fg\,d\mu\big|}{\|f\|_{F'}}.
\]
It follows that $\|g\|_{F^{(q)}}\le\|g\|_F\le u^{(q)}(F)\|g\|_{F^{(q)}}$.
\end{proof}
A different approach to the last result is needed when $q<1$ or $F$ is only assumed to be quasi-Banach. 
\begin{lemma} \label{ellu1F}
Let $0<q\le\infty$ and let $F$ be a quasi-Banach function lattice satisfying an upper $q$-estimate. Then there is a lattice quasi-norm $\|\cdot\|_{F^{(q)}}$ on $F$, such that

\begin{itemize}
\item[{\rm(i)}] $u^{(q)}(F^{(q)})=1$, where $F^{(q)}$ denotes $F$ with the quasi-norm $\|\cdot\|_{F^{(q)}}${\rm;} 
\item[{\rm(ii)}] the triangle inequality in $F^{(q)}$ holds with constant $\max\{2,2^{1/q}\}${\rm;} and
\item[{\rm(iii)}] for all $g\in F$, $\|g\|_{F^{(q)}}\le\|g\|_F\le u^{(q)}(F)\|g\|_{F^{(q)}}$.
\end{itemize}

\end{lemma}
\begin{proof} Again we prove only the case $q<\infty$. The same proof, with appropriate modifications to some of the expressions, works when $q=\infty$. 
For each $f\in F$, set
\[
\|f\|_{F^{(q)}}=\inf\Big(\sum_{j=1}^n\|f_j\|_F^q\Big)^{1/q},
\]
where the infimum is taken over all finite collections $f_1,\dots,f_n$ of disjointly supported functions that sum to $f$. Taking $n=1$ and $f_1=f$ shows that $\|f\|_{F^{(q)}}\le \|f\|_F$ and the definition of $u^{(q)}(F)$ shows that $\|f\|_F\le u^{(q)}(F)\|f\|_{F^{(q)}}$.  

Since $u^{(q)}(F)$ is assumed to be finite, we get $\|f\|_{F^{(q)}}\ge0$ with equality only if $f=0$ $\mu$-a.e.  For $\alpha\in\mathbb R$, it is routine to verify that $\|\alpha f\|_{F^{(q)}}=|\alpha|\|f\|_{F^{(q)}}$. To prove that $\|\cdot\|_{F^{(q)}}$ has the lattice property, suppose $f$ is a measurable function, $h\in F$ and $|f|\le|h|$. Since $F$ is a lattice, $f\in F$. Take $h_1,\dots,h_n$ to be disjointly supported functions in $F$ that sum to $h$ and take $H_1,\dots,H_n$ to be a measurable partition of the underlying measure space such that $h_j$ is supported on $H_j$ for all $j$. Then $f\chi_{H_1}+\dots+f\chi_{H_n}=f$ and $|f\chi_{H_j}|\le |h\chi_{H_j}|=|h_j|$ for all $j$. We have,
\[
\|f\|_{F^{(q)}}\le\Big(\sum_{j=1}^n\|f\chi_{H_j}\|_F^q\Big)^{1/q}\le\Big(\sum_{j=1}^n\|h_j\|_F^q\Big)^{1/q}.
\]
Taking the infimum over all such $h_1,\dots,h_n$ shows that $\|f\|_{F^{(q)}}\le\|h\|_{F^{(q)}}$.

We prove the triangle inequality in two steps. Suppose $f,g\in F$ have disjoint supports. Let $f_1,\dots, f_m$ be disjointly supported and sum to $f$ and let $g_1,\dots, g_n$ be disjointly supported and sum to $g$. Then $f_1,\dots,f_m,g_1,\dots,g_m$ are disjointly supported and sum to  $f+g$. Therefore, 
\[
\|f+g\|_{F^{(q)}}\le\Big(\sum_{j=1}^m\|f_j\|_F^q +\sum_{k=1}^n\|g_k\|_F^q\Big)^{1/q}.
\]
Taking the infimum over all such decompositions of $f$ and $g$, we get
\[
\|f+g\|_{F^{(q)}}\le\big(\|f\|_{F^{(q)}}^q+\|g\|_{F^{(q)}}^q\big)^{1/q}\le \max\{1,2^{1/q-1}\}\big(\|f\|_{F^{(q)}}+\|g\|_{F^{(q)}}\big).
\]

Now  we drop the disjoint support assumption on $f$ and $g$. Let $S_f$ be the set of points where $|f|\ge|g|$ and let $S_g$ be the set of points where $|g|>|f|$. Then $S_f$ and $S_g$ are disjoint, $|f+g|\chi_{S_f}\le2|f|$ and $|f+g|\chi_{S_g}\le 2|g|$, so
\begin{align*}
\|f+g\|_{F^{(q)}}&=\|(f+g)\chi_{S_f}+(f+g)\chi_{S_g}\|_{F^{(q)}}\\
&\le \max\{1,2^{1/q-1}\}\big(\|(f+g)\chi_{S_f}\|_{F^{(q)}}+\|(f+g)\chi_{S_g}\|_{F^{(q)}}\big)\\
&\le 2\max\{1,2^{1/q-1}\}\big(\|f\|_{F^{(q)}}+\|g\|_{F^{(q)}}\big).
\end{align*}
This proves that $\|\cdot\|_{F^{(q)}}$ is a quasi-norm and the triangle inequality in $F^{(q)}$ holds with constant $\max\{2,2^{1/q}\}$.

It remains to show that $u^{(q)}(F^{(q)})=1$. Suppose $f_1,\dots,f_n$ are disjointly supported and sum to $f$ in $F$. Fix $\theta>1$ and for each $j$, chose disjointly supported functions $f_{j,1},\dots, f_{j,n_j}$ that sum to $f_j$ and satisfy
\[
\theta\|f_j\|_{F^{(q)}}\ge\Big(\sum_{k=1}^{n_j}\|f_{j,k}\|_F^q\Big)^{1/q}.
\]
For each $j$, the supports of $f_{j,1},\dots, f_{j,n_j}$ are contained in the support of $f_j$ so $\{f_{j,k}:1\le k\le n_j; 1\le j\le n\}$ is a collection of disjointly supported functions that sum to $f$. Therefore,
\[
\|f\|_{F^{(q)}}\le\Big(\sum_{j=1}^n\sum_{k=1}^{n_j}\|f_{j,k}\|_F^q\Big)^{1/q}
\le\theta\Big(\sum_{j=1}^n\|f_j\|_{F^{(q)}}^q\Big)^{1/q}.
\]
Let $\theta\to1$ to see that $u^{(q)}(F^{(q)})\le 1$. But $u^{(q)}(F^{(q)})\ge1$ is trivial so this completes the proof.
\end{proof}

\section{A Christ-Kiselev theorem in quasi-Banach lattices}

In this section we give our main results, two closely related statements of the Christ-Kiselev maximal theorem for operators between quasi-Banach lattices. The first statement imposes a quantitative condition on two-term lower and upper estimates for the domain and range spaces and gives strong control of the constants involved. The second asks for the existence of upper and lower estimates, not just two-term estimates, but imposes no additional quantitative condition. Although the two theorems overlap, each may apply when the other does not. 

Throughout the section, $E$ and $F$ will be quasi-Banach function lattices over measure spaces $(\Omega,\mu)$ and $(\widetilde\Omega,\nu)$, respectively, and $T: E\to F$ will be a bounded, linear operator. We let $\kappa$ be the smallest constant for which the triangle inequality $\|f+g\|_F\le\kappa(\|f\|_F+\|g\|_F)$ holds. If $F$ is a Banach space, then $\kappa=1$.

A \emph{filtration} of $\Omega$ is a family $\mathcal A=\{A_\alpha:\alpha\in Q\}$ of measurable subsets of $\Omega$ indexed by a countable, totally ordered set $Q$ such that if $\alpha<\alpha'$, then $A_\alpha \subset A_{\alpha'}$.  
The \emph{maximal operator based on $T$ associated to the filtration $\mathcal A$} is the sublinear operator  $T^{*}$ defined by
\[
T^{*}f =\sup_{\alpha\in Q}|T(f\chi_{A_\alpha})|\quad\text{for all }f\in E. 
\]

First, the result with a quantitative hypothesis.
\begin{theorem} \label{thmChK2} Let $\mathcal A$ be a filtration of $\Omega$ and let $T^{*}$ be the maximal operator based on $T$ associated to $\mathcal A$. Assume that $F$ has the Fatou property. If there exist $p,q$ such that $0<p<q\le\infty$ and $\ell u<(1+\kappa^{-\tau})^{1/\tau}$, where $\ell =\ell_{(p), 2}(E)$, $u =u^{(q), 2}(F)$ and $\frac1\tau=\frac1p-\frac1q$, then 
$\|T^{*}\|_{E\to F}\le \gamma \|T\|_{E\to F}$, with
\[
\gamma=\frac{u(1+\kappa^{-\tau})^{1/q}}{\kappa^{-\tau/p}-(\ell^pu^p(1+\kappa^{-\tau})^{p/q}-1)^{1/p}}.
\]
If $F$ is a Banach space, then 
\[
\gamma=\frac{u2^{1/q}}{1-(\ell^pu^p2^{p/q}-1)^{1/p}}.
\]
\end{theorem} 

Next, the result with a qualitative hypothesis.
\begin{theorem} \label{thmChK}  Let $\mathcal A$ be a filtration of $\Omega$ and let $T^{*}$ be the maximal operator based on $T$ associated to the filtration $\mathcal A$. Assume that $F$ has the Fatou property. If there exist $p,q$ such that $0<p<q\le\infty$, $E$ satisfies a lower $p$-estimate and $F$ satisfies an upper $q$-estimate,  then  
$\|T^{*}\|_{E\to F}\le \delta \|T\|_{E\to F}$, where
\[
\delta=\frac{(1+\kappa_q^{-\tau})^{1/q}}{\kappa_q^{-\tau/p}-((1+\kappa_q^{-\tau})^{p/q}-1)^{1/p}}\ell_{(p)}(E)u^{(q)}(F).
\]
Here $\frac1\tau=\frac1p-\frac1q$ and $\kappa_q=\max\{2,2^{1/q}\}$. If $F$ is a Banach space and $q\ge1$, then 
\[
\delta=\frac{2^{1/q}}{1-(2^{p/q}-1)^{1/p}}\ell_{(p)}(E)u^{(q)}(F).
\]
\end{theorem} 

The proof of Theorems \ref{thmChK2} and \ref{thmChK} will be given after Theorem \ref{lemmaCK} and Corollary \ref{lpuq}, below. But first we look at the following special case, which extends the Christ-Kiselev maximal theorem to Lebesgue indices less than~$1$. 

It is important to point out that when $0<p<1$ and $p<q<\infty$, any bounded, linear operator $T:L^p_\mu\to L^q_\nu$ necessarily maps functions supported on the non-atomic part of $\mu$ to zero. So the case $0<p<1$ of the following extension is only of interest when $\mu$ is a purely atomic measure.

\begin{corollary}\label{pq} Let $(\Omega,\mu)$ and $(\widetilde\Omega,\nu)$ be measure spaces. Suppose $0<p<q\le\infty$ and $T$ is a bounded, linear operator from $L^p_\mu$ to $L^q_\nu$. Let $\mathcal A$ be a filtration of $\Omega$ and let $T^{*}$ be the maximal operator based on $T$ associated to the filtration $\mathcal A$. Then $\|T^{*}\|\le \gamma\|T\|$,
with
\[
\gamma=\begin{cases}\displaystyle\frac{2^{1/q}}{1-(2^{p/q}-1)^{1/p}},&\text{ if $q\ge1$};\\ 
\displaystyle\frac{(1+2^{(q-1)p/(q-p)})^{1/q}}{2^{(q-1)/(q-p)} - ((1+2^{(q-1)p/(q-p)})^{p/q}-1)^{1/p}},&\text{ if $q<1$}.\end{cases}
\]
\end{corollary}
\begin{proof} Observe that $\ell_{(p), 2}(L^p_\mu)=u^{(q), 2}(L^q_\nu)=1$ and the triangle inequality in $L^q_\nu$ holds with constant $\kappa=\max(1,2^{(1/q)-1})$. The result follows from  Theorem \ref{thmChK2}.
\end{proof}
\begin{remark}\label{const} The constant $\gamma=\frac{2^{1/q}}{1-(2^{p/q}-1)^{1/p}}$, in the case $1\le p<q\le\infty$ above, is less than $(1-2^{1/q-1/p})^{-1}$, the corresponding constant from {\rm Theorem \ref{CKorig}.} One way to see this is to let $h(t)=t+2^{1/p}t^{-1} -(2-t^p)^{1/p}$ for $t\in[1,2^{1/p}]$. The form of $h'(t)=1-2^{1/p}t^{-2} +t^{p-1}(2-t^p)^{-1/p'}$ shows directly that $h'(t)>h'(1)\ge0$ for all $t\in(1,2^{1/p})$. Thus, $h$ is strictly increasing and so $h(2^{1/p-1/q})>h(1)$, which gives the result. 

In addition, $\gamma\to1$ as $q\to\infty$, but $(1-2^{1/q-1/p})^{-1}\to(1-2^{-1/p})^{-1}$, which is not bounded above as $p\to\infty$. 
\end{remark}

The key induction argument needed to prove Theorems  \ref{thmChK2} and \ref{thmChK} is isolated in the proof of the next result, which follows the method of  \cite[Theorem 8.7]{Tao}. The most important feature, for our purposes, is that the constant $\gamma$ does not depend on $n$. 

\begin{theorem} \label{lemmaCK} Let $\mu$, $\nu$, $E$, $F$, $T$ and $\kappa$ be as above. If there exist $p$ and $q$ such that $0<p<q\le\infty$ and $\ell u<(1+\kappa^{-\tau})^{1/\tau}$, where $\ell =\ell_{(p), 2}(E)$, $u =u^{(q), 2}(F)$ and $\frac1\tau=\frac1p-\frac1q$, then for each positive integer $n$, 
\[
\Big\|\sum_{1\le j\le k\le n}\chi_{\widetilde\Omega_k}T(f\chi_{\Omega_j})\Big\|_F\le \gamma \,\|T\|_{E\to F}\|f\chi_{\Omega_1\cup\dots\cup\,\Omega_n}\|_E, \quad\, f\in E\,,
\]
whenever $\{\Omega_j\}_{j=1}^n$ are disjoint measurable subsets of $\Omega$ and $\{\widetilde\Omega_j\}_{j=1}^n$ are disjoint measurable subsets of $\widetilde\Omega$. Here 
\[
\gamma=\frac{u(1+\kappa^{-\tau})^{1/q}}{\kappa^{-\tau/p}-(\ell^pu^p(1+\kappa^{-\tau})^{p/q}-1)^{1/p}}.
\]

\end{theorem}
\begin{proof} Without loss of generality, assume that $\|T\|_{E\to F}=1$. Fix $\beta\in(0,1)$ and $\gamma\ge 1$, both to be determined later. We proceed by induction. If $n=1$, $\Omega_1\subset \Omega$, $\widetilde\Omega_1\subset \widetilde\Omega$, and $f\in E$, then
\[
\|\chi_{\widetilde\Omega_1}T(f\chi_{\Omega_1})\|_F\le\|T(f\chi_{\Omega_1})\|_F \le \|f\chi_{\Omega_1}\|_E
\le \gamma \|f\chi_{\Omega_1}\|_E.
\]

Now suppose $n\ge2$ and the statement holds (with our fixed $\gamma$) for all positive integers smaller than $n$. To prove that it also holds for $n$, fix $\Omega_1,\dots,\Omega_n$ and $\widetilde\Omega_1,\dots,\widetilde\Omega_n$ as above and let $f\in E$, scaled so that $\|f\chi_{\Omega_1\cup \dots \cup\,\Omega_n}\|_E=1$. (The conclusion is trivial if $f=0$ a.e. on $\Omega_1\cup \dots \cup\,\Omega_n$.) Observe that
\[
0=\|f\chi_\emptyset\|_E\le \|f\chi_{\Omega_1}\|_E
\le\|f\chi_{\Omega_1\cup\Omega_2}\|_E\le\dots\le \|f\chi_{\Omega_1\cup\dots\cup\Omega_n}\|_E=1.
\]
Choose $m$ with $1\le m \le n$ to satisfy
\[
\|f\chi_{\Omega_1\cup\dots\cup\Omega_{m-1}}\|_E < \beta \le \|f\chi_{\Omega_1\cup\dots \cup \Omega_m}\|_E.
\]
By the definition of $\ell =\ell_{(p),2}$, we get
\[
\|f\chi_{\Omega_{m+1}\cup\dots\cup\Omega_n}\|^p_E
\le \ell^p \|f\chi_{\Omega_1\cup\dots\cup \,\Omega_n}\|^p_E-\|f\chi_{\Omega_1 \cup \dots \cup \,\Omega_m}\|^p_E \le \ell^p-\beta^p\,
\]
so we may apply the inductive hypothesis to the sequences $\{\Omega_{m+1},\dots,\Omega_n\}$ and $\{\widetilde\Omega_{m+1},\dots,\widetilde\Omega_n\}$ to get
\[
\Big\|\sum_{m<j\le k\le n}\chi_{\widetilde\Omega_k}T(f\chi_{\Omega_j})\Big\|_F
\le \gamma \|f\chi_{\Omega_{m+1}\cup\dots\cup \,\Omega_n}\|_E\le \gamma(\ell^p-\beta^p)^{1/p}.
\]
Additionally, the linearity of $T$ implies
\[
\Big\|\sum_{1\le j\le m\le k\le n}\chi_{\widetilde\Omega_k}T(f\chi_{\Omega_j})\Big\|_F
=\|\chi_{\widetilde\Omega_m\cup\dots\cup\widetilde\Omega_n}T(f\chi_{\Omega_1\cup\dots\cup \,\Omega_m})\|_F \le 1.
\]
These two estimates, and the triangle inequality in $F$, imply
\[
\Big\|\sum_{1\le j\le k\le n,\, k\ge m}\chi_{\widetilde\Omega_k}T(f\chi_{\Omega_j})\Big\|_F\le\kappa(1+\gamma(\ell^p-\beta^p)^{1/p}).
\]
Applying the inductive hypothesis to $\{\Omega_1,\dots,\Omega_{m-1}\}$ and $\{\widetilde\Omega_1,\dots,\widetilde\Omega_{m-1}\}$, we get
\[
\Big\|\sum_{1\le j\le k< m}\chi_{\widetilde\Omega_k}T(f\chi_{\Omega_j})\Big\|_F \le \gamma \|f\chi_{\Omega_1\cup\dots\cup \Omega_{m-1}}\|_E
\le  \gamma\beta.
\]
Now we proceed in two cases. 

First suppose $q<\infty$. The sums in the last two inequalities yield disjointly supported functions so the upper $q$-estimate for $F$ yields
\[
\Big\|\sum_{1\le j\le k\le n}\chi_{\widetilde\Omega_k}T(f\chi_{\Omega_j})\Big\|_F^q
\le u^q((\gamma\beta)^q +\kappa^q(1+\gamma(\ell^p-\beta^p)^{1/p})^q).
\]
If $\beta$ and $\gamma$ can be chosen so that the last expression equals $\gamma^q$, then the induction is complete. Setting $u^q((\gamma\beta)^q +\kappa^q(1+\gamma(\ell^p-\beta^p)^{1/p})^q)=\gamma^q$, we isolate $\gamma$ as
\[
\gamma=\frac {u\kappa }{(1-(u\beta)^q)^{1/q}-\kappa((\ell u)^p-(u\beta)^p)^{1/p}}.
\]
Setting $t=u\beta$, the first term in the denominator imposes the restriction $0<t\le1$ and, to ensure that $\gamma>0$, we get the additional restriction $\ell u<(t^p+\kappa^{-p}(1-t^q)^{p/q})^{1/p}$. Note that $\gamma\ge1$ whenever it is positive. We want to impose the mildest possible condition on $u$ and $\ell$, which means finding the maximum of the right hand side. Calculus shows that it occurs at $t=(1+\kappa^{-\tau})^{-1/q}$ and gives the condition $\ell u<(1+\kappa^{-\tau})^{1/\tau}$. So the best choice of $\beta$ is $u^{-1}(1+\kappa^{-\tau})^{-1/q}$, which is clearly less than $1$. Putting it into the expression for $\gamma$ gives the statement of the theorem. 

Next consider the case $q=\infty$. This time the $q$-estimate yields
\[
\Big\|\sum_{1\le j\le k\le n}\chi_{\widetilde\Omega_k}T(f\chi_{\Omega_j})\Big\|_F
\le \max\{u\gamma\beta,u\kappa(1+\gamma(\ell^p-\beta^p)^{1/p})\}.
\]
We need a $\beta\in(0,1)$ and a $\gamma\ge1$ for which this expression is at most $\gamma$. The obvious choice, $\beta=1/u$, may violate the $\beta<1$ requirement so instead we note that $\tau=p$ and use the hypothesis $\ell^p u^p<1+\kappa^{-p}$ to choose $\beta$ so that $\ell^pu^p-\kappa^{-p}<u^p\beta^p<1$. Clearly, $0<\beta<1$. Now we set the second term in the maximum equal to $\gamma$ and solve to get
\[
\gamma=\frac {u\kappa }{1-\kappa((\ell u)^p-(u\beta)^p)^{1/p}}.
\]
Since $u\beta<1\le\ell u$, the denominator lies between $0$ and $1$. But both $u$ and $\kappa$ are at least $1$, so $\gamma\ge1$. The choice of $\beta$ shows that the first term in the maximum is less than $\gamma$. This completes the induction. Now we let $\beta$ increase to $1/u$ and observe that $\gamma$ tends to the constant in the statement of the theorem.
\end{proof}

If $E$ satisfies a lower $p$-estimate and $F$ satisfies an upper $q$-estimate, then Lemmas \ref{ellu1E} and \ref{ellu1F} give us equivalent quasi-norms under which both estimates have constant $1$. Equipped with the new quasi-norms, $E$ and $F$ satisfy the condition of Lemma \ref{lemmaCK} to give the following result.

\begin{corollary} \label{lpuq} Let $\mu$, $\nu$, $E$, $F$, $T$ and $\kappa$ be as above. If there exist $p,q$ such that $0<p<q\le\infty$, $E$ satisfies a lower $p$-estimate and $F$ satisfies an upper $q$-estimate, then for each positive integer $n$, 
\[
\Big\|\sum_{1\le j\le k\le n}\chi_{\widetilde\Omega_k}T(f\chi_{\Omega_j})\Big\|_F\le \delta\|T\|_{E\to F}\|f\chi_{\Omega_1\cup\dots\cup\,\Omega_n}\|_E, \quad f\in E,
\]
whenever $\{\Omega_j\}_{j=1}^n$ are disjoint measurable subsets of $\Omega$ and $\{\widetilde\Omega_j\}_{j=1}^n$ are disjoint measurable subsets of $\widetilde\Omega$. Here 
\[
\delta=\frac{(1+\kappa_q^{-\tau})^{1/q}}{\kappa_q^{-\tau/p}-((1+\kappa_q^{-\tau})^{p/q}-1)^{1/p}}\ell_{(p)}(E)u^{(q)}(F),
\]
where $\kappa_q=\max\{2,2^{1/q}\}$ and $\frac1\tau=\frac1p-\frac1q$. If $F$ is a Banach space and $q\ge1$, then
\[
\delta=\frac{2^{1/q}}{1-(2^{p/q}-1)^{1/p}}\ell_{(p)}(E)u^{(q)}(F)
\]
also works.
\end{corollary}
\begin{proof} Let $E_{(p)}$ and $F^{(q)}$ be the spaces introduced in Lemmas \ref{ellu1E} and \ref{ellu1F} and note that the triangle inequality in $F^{(q)}$ holds with constant $\kappa_q=\max\{2,2^{1/q}\}$. Also note that $u^{(q)}(F^{(q)})=\ell_{(p)}(E_{(p)})=1$. Let
\[
\mathcal Tf =\sum_{1\le j\le k\le n}\chi_{\widetilde\Omega_k}T(f\chi_{\Omega_j}), \quad f\in E.
\]
Replacing $E$ by $E_{(p)}$ and $F$ by $F^{(q)}$ in Theorem \ref{lemmaCK} we get
\[
\|\mathcal Tf\|_{F^{(q)}}\le\delta_0\|f\|_{E_{(p)}},
\]
where $\delta_0$ is the value of $\gamma$ when $u=\ell=1$ and $\kappa=\kappa_q$. Returning to the original quasi-norms we have
\[
\|\mathcal Tf\|_F\le u^{(q)}(F)\|\mathcal Tf\|_{F^{(q)}}\le u^{(q)}(F)\delta_0\|f\|_{E_{(p)}}\le \delta_0\ell_{(p)}(E)u^{(q)}(F)\|f\|_E.
\] 
This proves the first statement, with $\delta=\delta_0\ell_{(p)}(E)u^{(q)}(F)$. 

If $F$ is a Banach space and $q\ge1$, we may use the space $F^{(q)}$ from Lemma \ref{comb} instead and repeat the above argument with $\kappa=1$. This proves the second statement.
\end{proof}

We are now in a position
to prove Theorems $\ref{thmChK2}$ and  $\ref{thmChK}$ based on Theorem \ref{lemmaCK} and Corollary \ref{lpuq}. The presentation follows  \cite[Corollary 8.8]{Tao}.

\begin{proof}[Proof of Theorems $\ref{thmChK2}$ and  $\ref{thmChK}$]  We will prove Theorem $\ref{thmChK2}$ only. Theorem $\ref{thmChK}$ follows from Corollary \ref{lpuq} in exactly the same way as Theorem $\ref{thmChK2}$ follows from Theorem \ref{lemmaCK}.

Suppose $Q$ is a finite set, say
$Q=\{\alpha_1,\dots,\alpha_n\}$ with $\alpha_1 <\alpha_2 <\dots <\alpha_n$. Fix $f\in E$ and for every $y\in \widetilde\Omega$
choose  $\alpha(y)\in Q$ so that
\[
T^*f(y)=|T(f\chi_{A_{\alpha(y)}})(y)|.
\]
In $\Omega$, define $\Omega_1,\dots,\Omega_n$ by setting $\Omega_1=A_{\alpha_1}$ and $\Omega_j=A_{\alpha_j}\setminus A_{\alpha_{j-1}}$ for $j=2,\dots,n$. In $\widetilde\Omega$, set 
$\widetilde\Omega_k=\{y\in \widetilde\Omega: \alpha(y)=\alpha_k\}$ for each $k\in \{1,\dots, n\}$. We will apply Theorem \ref{lemmaCK}  to these sequences of disjoint sets. Note that each $\Omega_j$ is $\mu$-measurable and each $\widetilde\Omega_k$ is $\nu$-measurable.

The choice of $\widetilde\Omega_1,\dots,\widetilde\Omega_n$ ensures that, for each $y\in\widetilde\Omega$,
\[
T(f\chi_{A_{\alpha(y)}})(y)=\sum_{k=1}^n\chi_{\widetilde\Omega_k}(y)T(f\chi_{A_{\alpha_k}})(y).
\]
For each $k$,  $A_{\alpha_k}$ is the disjoint union of $\Omega_1,\dots,\Omega_k$, and $T$ is linear, so  
\[
T(f\chi_{A_{\alpha_k}})=T\Big( f\sum_{j=1}^k\chi_{\Omega_j}\Big)=\sum_{j=1}^kT(f\chi_{\Omega_j}).
\]
Combining these, we obtain
\[
T^{*}f=\Big|\sum_{1\le j\le k\le n}\chi_{\widetilde\Omega_k}T(f\chi_{\Omega_j})\Big|.
\]
Since $f$ was arbitrary in $E$, Theorem \ref{lemmaCK} shows that 
\[
\|T^{*}\|_{E\to F}\le \gamma\|T\|_{E\to F}.
\]

Next suppose $Q$ is countably infinite. Express $Q$ as the union of an increasing sequence of finite subsets $Q_m$ to get, for each $y\in\widetilde\Omega$,
\[
T^{*}f(y)=\sup_{\alpha\in Q}|T(f\chi_{A_\alpha})(y)|=\lim_{m\to\infty}\sup_{\alpha\in Q_m}|T(f\chi_{A_\alpha})(y)|,
\]
an increasing limit. The Fatou property of $F$ ensures that
\[
\|T^{*}f\|_F=\lim_{m\to\infty}\big\|\sup_{\alpha\in Q_m}|T(f\chi_{A_\alpha})|\big\|_F
\le \gamma\|T\|_{E\to F}\|f\|_E.
\]
This completes the proof.
\end{proof}

We conclude with a remark similar to one in \cite{ChK1}. Because the constants appearing in the conclusions of Theorems \ref{thmChK2} and \ref{thmChK} are independent of the countable collection $\{A_\alpha\}$, a corresponding result can be deduced for continuum filtration, indexed by a real variable, as a direct consequence via a limiting argument.

\section{In Lorentz spaces}

Here we apply the results of the previous section to establish a Christ-Kiselev maximal theorem in Lorentz spaces. It is interesting to observe how straightforward convexity conditions on the Lorentz weight function imply the upper $q$- and lower $p$-estimates needed to apply Theorem \ref{thmChK2}.

Given a $\sigma$-finite measure space $(\Omega, \mu)$, an index $r$ with $0<r<\infty$, and a nonnegative measurable, ``weight'' function $w$ on $(0,\infty)$, the Lorentz space $\Lambda_{r, w}= \Lambda_{r, w}(\Omega,\mu)$ is defined to be the set of all $f\in L^0_\mu $ such that
\[
\|f\|_{r,w}  = \bigg(\int_0^\infty (f^{*})^r w\bigg)^{1/r} <\infty.
\]
It is well known that $\Lambda_{r,w}$ is a rearrangement invariant Banach space if $1\le r<\infty$ and $w$ is nonincreasing. Also, $\Lambda_{r,w}$ is a rearrangement invariant quasi-Banach space if $0<r<\infty$ and for some $C>0$, $\int_0^{2t}w\le C\int_0^tw$ for $t>0$. (See \cite[Theorem A]{KM} for conditions under which $\Lambda_{r,w}$ can be re-normed to become a Banach space.)

In the case $0<r, p<\infty$ and $w(t) = \frac{r}{p} t^{\frac rp-1}$ for all $t>0$, $\Lambda_{r, w}$ coincides with the classical Lorentz space $L_{p, r}=L_{p, r}(\Omega, \mu)$. Note that if $r>p$, then this weight function is increasing and the triangle inequality fails for $\|\cdot\|_{r, w}$. 

The convexity and concavity of the classical Lorentz spaces $L_{r, p}$ over a non-atomic measure space with $1<r<\infty$, $1\le p<\infty$ were studied by Creekmore \cite{Cr}. A study of order convexity and concavity in the context of quasi-Banach Lorentz spaces $\Lambda_{r,w}$ was carried out in \cite{KM}; the results were stated for Lebesgue measurable functions only, but we will use their methods in the next lemma. See \cite[Theorems 3 and 7]{KM} in particular.

\begin{theorem}\label{Lorentz} Let $(\Omega, \mu)$ and $(\widetilde\Omega, \nu)$ be measure spaces. Suppose $0<r<s<\infty$ and $w$ and $v$ are weight functions. Set $W(t)=\int_0^tw$ and $V(t)=\int_0^tv$ for $t>0$. If $T: \Lambda_{r, w}(\Omega, \mu) \to \Lambda_{s, v}(\widetilde\Omega, \nu)$ is bounded and there exist $p$ and $q$ with $r\le p<q\le s$ such that $t\mapsto W(t^{p/r})$ is a convex function and  $t\mapsto V(t^{q/s})$ is a concave function, then $T^*$, the maximal operator based on $T$ associated to filtration, is bounded from $\Lambda_{r, w}(\Omega, \mu)$ to $\Lambda_{s, v}(\widetilde\Omega, \nu)$ and $\|T^{*}\| \leq C \|T\|$, where $C$ is a constant depending only on $p$, $q$ and the triangle inequality constant for $\Lambda_{s, v}$. If $\Lambda_{s, v}$ is a Banach space, then $C=\frac{2^{1/q}}{1-(2^{p/q}-1)^{1/p}}$.
\end{theorem}
\begin{proof} Since $W(t^{p/r})$ is a convex function of $t\ge0$ and $W(0)=0$, $W(t^{p/r})/t$ is nondecreasing in $t$ and so are $W(t)/t^{r/p}$ and $W(t)^{p/r}/t$. Therefore, if $t=t_1+\dots+t_n$ is a sum of nonnegative numbers, then
\[
W(t)=\Big(\sum_{j=1}^n(t_j/t)W(t)^{p/r}\Big)^{r/p}\ge \Big(\sum_{j=1}^nW(t_j)^{p/r}\Big)^{r/p}.
\]
Since $f^*$ and $f$ are equimeasurable, if $y>0$, then $\{t>0:f^*(t)>y\}=\{t>0:\mu_f(y)>t\}=(0,\mu_f(y))$. Using the Fubini theorem we see that for all $f\in \Lambda_{r,w}$, 
\begin{align*}
\|f\|_{r,w}^r &= \int_o^\infty\int_0^{f^*(t)}\,d(y^r) w(t)\,dt
\\&= \int_0^\infty\int_{\{t>0:f^*(t)>y\}}w(t)\,dt\,d(y^r) =\int_0^\infty W(\mu_f(y))\,d(y^r). 
\end{align*}
Now, if $f_1,\dots, f_j$ are disjointly supported functions that sum to $f$, then $\mu_f=\mu_{f_1}+\dots+\mu_{f_n}$ so, following the above observations with Minkowski's inequality, we get 
\begin{align*}
\|f\|_{r,w}&\ge \bigg(\int_0^\infty \Big(\sum_{j=1}^nW(\mu_{f_j}(y))^{p/r}\Big)^{r/p}\,d(y^r)\bigg)^{1/r}\\
&\ge \bigg(\sum_{j=1}^n\bigg(\int_0^\infty W(\mu_{f_j}(y))\,d(y^r)\bigg)^{p/r}\bigg)^{1/p}
=\Big(\sum_{j=1}^n \|f_j\|_{r,w}^p\Big)^{1/p}.
\end{align*}
This shows that $\Lambda_{r, w}$ has a lower $p$-estimate and $\ell_{(p)}(\Lambda_{r, w})=1$. 

A similar argument shows that $\Lambda_{s, v}$ has an upper $q$-estimate and that $u^{(q)}(\Lambda_{s, v})=1$. It follows that  $\ell_{(p),2}(\Lambda_{r, w})=u^{(q),2}(\Lambda_{s, v})=1$ so Theorem \ref{thmChK2} proves the boundedness of $T^*$ and gives a formula for the constant $C$ that has the stated properties.
\end{proof}
 An examination of the proof reveals that we may replace the convexity and concavity assumptions in the previous result by the following weaker hypotheses: For all $t_1,t_2\ge0$,
\[
W(t_1+t_2)^{p/r}\ge W(t_1)^{p/r}+W(t_2)^{p/r}
\ {and}\ 
V(t_1+t_2)^{q/s}\le V(t_1)^{q/s}+W(t_2)^{q/s}.
\]

A Christ-Kiselev maximal theorem was given in \cite[Theorem 1.2]{CKL} for classical Lorentz spaces with indices in $[1,\infty)$. As a consequence of Theorem \ref{Lorentz}, we obtain a result for the classical Lorentz spaces with indices in $(0,\infty)$. It also improves the constant. See Remark \ref{const}.
\begin{corollary} \label{ChKLorentz}
Fix indices satisfying $0<p\le r <s \le q<\infty$ and let $T: L_{p,r}(\Omega, \mu) \to L_{q, s}(\widetilde\Omega, \nu)$ be a bounded linear operator between classical Lorentz spaces. Then $T^{*}$, the maximal operator based on $T$ associated to a filtration, is bounded from $L_{p,r}(\Omega, \mu)$ to $L_{q, s}(\widetilde\Omega, \nu)$ with $\|T^{*}\| \le C\|T\|$, where $C$ is a constant depending only on $p$, $q$ and $s$. If $s\ge1$, we may take $C=\frac{2^{1/q}}{1-(2^{p/q}-1)^{1/p}}$.
\end{corollary}
\begin{proof} With $\Lambda_{r,w}=L_{p,r}$ and $\Lambda_{s,v}=L_{q,s}$ we have $W(t)=t^{r/p}$ and $V(t)=t^{s/q}$. Thus $W(t^{p/r})=t=V(t^{q/s})$ so the convexity and concavity conditions hold trivially and we get $\|T^{*}\| \le C\|T\|$ for some $C$ depending on $p$, $q$ and the constant in the triangle inequality for $L_{q,s}$. The latter depends only on $q$ and $s$. If $s\ge1$ then $L_{q,s}$ is a Banach space, because $s\le q$.
\end{proof}

\subsection{Application: The Fourier Transform in Lorentz spaces}

Easy-to-verify weight conditions are available that imply boundedness of the Fourier transform from one Lorentz space to another. These can be combined with the Christ-Kiselev maximal inequality given above. We illustrate this with a simple special case. 

Here the measure spaces $(\Omega,\mu)$ and $(\widetilde\Omega, \nu)$ are both just $\mathbb R^n$ with Lebesgue measure.

\begin{corollary}\label{FTLor} Let $1<r<s<\infty$ with $s\ge2$ and let $v$ and $w$ be nonincreasing weights on $(0,\infty)$ such that
\[
\sup_{z>0}zV(1/z)^{1/r}W(z)^{-1/s}<\infty.
\]
Here $W(t)=\int_0^tw$ and $V(t)=\int_0^tv$ for $t>0$.

If there exists an index $p$, with $r<p<s$, such that $t\mapsto W(t^{p/r})$ is convex, then $\mathcal F^*$, the maximal operator based on the Fourier transform and associated to a filtration of Lebesgue measurable subsets of $\mathbb R^n$, is bounded from $\Lambda_{r,w}$ to $\Lambda_{s,v}$. 
\end{corollary}
\begin{proof} From \cite[Definition 1 and Theorem 2]{BH}, the above conditions on $r$, $s$, $v$, and $w$ ensure that the Fourier transform is bounded from $\Lambda_{r,w}$ to $\Lambda_{s,v}$. Choose $p$ as in the statement. With $q=s$, $t\mapsto V(t^{q/s})$ is concave because $v$ is nonincreasing. Theorem \ref{Lorentz} now gives the result. 
\end{proof}

This result is not vacuous; the case $v=w=1$, $s=r'$ recovers the Christ-Kiselev maximal theorem applied to the Hausdorff-Young inequality.

\subsection{Application: Consequences of a result of Sledd and Stegenga}

A deep result of Sledd and Stegenga, in \cite{SlSt}, characterizes the measures on $\mathbb R^n\setminus\{0\}$ with respect to which the Fourier transform of every $H^1$ function is integrable. Here $H^1=H^1(\mathbb R^n)$ is the real variable Hardy space of \cite{FeSt}. With this as our starting point, we employ interpolation to produce a class of Fourier inequalities in weighted Lebesgue spaces, a class of Fourier inequalities between classical Lorentz spaces involving absolutely continuous measures on $\mathbb R^n$, and a subset of the latter for which the Fourier inequality can be strengthened to an inequality for any maximal operator based on the Fourier transform and associated to a filtration.

For notational convenience, if $u$ is a positive Lebesgue measurable function on $\mathbb R^n$ we let $L^p(u)=L^p_\mu$ and $L_{p,r}(u)=L_{p,r}(\mathbb R^n,\mu)$, where the measure $\mu$ is defined by $\mu(E)=\int_Eu$ on every Lebesgue measurable set $E$.

\begin{theorem}\label{SS} Let $w$ be a positive, Lebesgue measurable function on $\mathbb R^n$ and let $w_t(x)=t^nw(x/t)$ for all $t>0$. Suppose that
$$
\sup_{t>0}\sum_{k\in\mathbb Z^n\setminus\{0\}}\Big(\int_{k+[-1/2,1/2]^n}w_t\Big)^2<\infty.
$$
Then $\mathcal F:H^1\to L^1(w)$ is bounded and, if $1<p<2$, then
$$
\mathcal F:L^p\to L^p(w^{2-p})\quad\text{is bounded}.
$$
If $1<p<q<p'$ and  $1\le r\le s\le\infty$, then 
$$
\mathcal F:L_{p,r}\to L_{q,s}(w^{(p'-q)/p'})\quad\text{is bounded}. 
$$
Finally, if $1<p\le r< s\le q<p'$, then for any maximal operator $\mathcal F^*$ based on $\mathcal F$ and associated to a filtration,
$$
\mathcal F^*:L_{p,r}\to L_{q,s}(w^{(p'-q)/p'})\quad\text{is bounded.} 
$$

\end{theorem}
\begin{proof}
Theorem 10.4 of \cite{FS} gives a special case of \cite{SlSt} in which the Borel measure is replaced by weighted Lebesgue measure with weight $w$, showing that for all $f\in H^1$,
$$
\int_{\mathbb R^n}|\hat f|w<\infty.
$$
The closed graph theorem shows that $\mathcal F:H^1\to L^1(w)$ is bounded. 

Since $\mathcal F:L^2\to L^2$ is bounded, we may employ complex interpolation to get the boundedness of $\mathcal F:[H^1,L^2]_{2/p'}\to[L^1(w),L^2]_{2/p'}$. In the comments after \cite[Theorem 1]{JaJo}, we see that $[H^1,L^2]_{2/p'}=L^p$. Since $w$ is positive, $L^1(w)$ is a space of Lebesgue measurable functions. Thus, \cite[Theorem 5.5.3]{BL} shows $[L^1(w),L^2]_{2/p'}=L^p(w^{2-p})$. Thus $\mathcal F:L^p\to L^p(w^{2-p})$ is bounded.

For the Lorentz space results, we take $p$ and $q$ with $1<p<q<p'$, set $\theta=\frac1{p'}+\frac1q$, and observe that $0<\theta<1<\theta q<2$. Applying the Lebesgue space result just proved, with $p$ replaced by $\theta q$, we find that 
$$
\mathcal F:L^{\theta q}\to L^{\theta q}(w^{(p'-q)/p'}).
$$
Since $w$ is positive, $L^\infty=L^{\infty}(w^{(p'-q)/p'})$ with identical norms. Thus,
$$
\mathcal F:L^1\to L^{\infty}(w^{(p'-q)/p'}).
$$
The choice of $\theta$ makes $\frac{1-\theta}1+\frac\theta{\theta q}=\frac1p$ and $\frac{1-\theta}\infty+\frac\theta{\theta q}=\frac1q$, so two applications of \cite[Theorem 5.5.1]{BL} show that
$$
\mathcal F:L_{p,r}\to L_{q,r}(w^{(p'-q)/p'}).
$$
But $r\le s$ so $L_{q,r}(w^{(p'-q)/p'})\hookrightarrow L_{q,s}(w^{(p'-q)/p'})$, which gives the result.

The boundedness of the maximal Fourier transform is a direct application of Corollary \ref{ChKLorentz} to the last result; one simply ensures that all restrictions on the indices are met.
\end{proof}

\section{In Weiner amalgams}

In investigations of boundedness of the Fourier transform between Banach function lattices, the Wiener amalgam spaces arose naturally as appropriate domains, providing necessary conditions for Fourier estimates and pushing the theory beyond the scale of Lebegue spaces nearly a century ago. Their continued development is related in \cite{F} but we will not require the generality discussed there. Let us recall the definition of the amalgam of a Lebesgue function space and a Lebesgue sequence space. Given $1 \le r, s< \infty$ the Wiener amalgam space $W_n(L^r, \ell^s)$ is the Banach function lattice over $\mathbb R^n$ equipped with the norm
\[
\|f\|_{W_n(L^r, \ell^s)}  = \bigg(\sum_{k\in \mathbb{Z}^n} \bigg(\int_{Q_k} |f|^r\bigg)^{s/r}\bigg)^{1/s}.
\]
Here $Q_k$ denotes the translate by $k\in \mathbb{Z}^n$ of the unit cube $[0, 1)^n$. The usual modifications are made to define $W_n(L^r, \ell^s)$ in case $r$ or $s$ is infinite. Evidently, $W_n(L^r, \ell^r)=L^r$ with identical norms. Observe that the space decreases as $r$ increases
and increases as $s$ increases. Wiener studied the spaces $W_1(L^1, \ell^2)$ and $W_1(L^2, \ell^{\infty})$ in \cite{W1} and the spaces $W_1(L^{\infty}, \ell^2)$ and $W_1(L^1, \ell^{\infty})$ in \cite{W2}. 

For a detailed discussion of Wiener amalgams of Lebesgue function and sequence spaces (over locally compact abelian groups) we refer to Fournier and Stewart's paper \cite{FS}. The amalgam version of the Hausdorff-Young inequality, stated in their Theorem 2.18, asserts that if $1 \le r \le 2$ and $1 \le s\le 2$, then
\[
\|\mathcal F\|_{W_n(L^r, \ell^s)\to W_n(L^{s'}, \ell^{r'})}<\infty.
\]
The largest of the spaces $(L^r, \ell^s)$ for $1\le r, s\le 2$ is $W_n(L^1, \ell^2)$. It is a remarkable result proved by Szeptycki in \cite{Sz} that this is the largest Banach function lattice over $\mathbb{R}^n$ that is mapped into a space of measurable functions by the Fourier transform.

\begin{lemma}\label{ulamalg} Let $p,q,r,s\in(0,\infty]$. Then
\[
\ell_{(p),2}(W_n(L^r,\ell^s))=2^{\frac1p-\frac1{\max\{p,r,s\}}}
\text{ and } 
u=u^{(q),2}(W_n(L^{r},\ell^{s}))=2^{\frac1{\min\{q,r,s\}}-\frac1q}.
\]
\end{lemma}
\begin{proof} Let $\mathbb Z_2=\{0,1\}$. Suppose the functions $f_0,f_1\in W_n(L^r,\ell^s)$ are disjointly supported and define $a=(a_{i,k})_{i\in\mathbb Z_2;k\in\mathbb Z^n}$ by $a_{i,k}^r=\int_{Q_k}|f_i|^r$. Disjointness shows that $a_{0,k}^r+a_{1,k}^r=\int_{Q_k}|f_0+f_1|^r$. For any constant $C$, the inequality 
\[
\big(\|f_0\|_{W_n(L^r,\ell^s)}^p+\|f_1\|_{W_n(L^r,\ell^s)}^p\big)^{1/p}\le C\|f_0+f_1\|_{W_n(L^r,\ell^s)},
\]
appropriately modified if $p=\infty$, holds for all such $f_0$ and $f_1$ if and only if the mixed norm embedding
\[
\|\|a\|_{\ell^s(\mathbb Z^n)}\|_{\ell^p(\mathbb Z_2)}\le C\|\|a\|_{\ell^r(\mathbb Z_2)}\|_{\ell^s(\mathbb Z^n)}
\]
holds for all non-negative $a$. By \cite[Theorem 4.1]{GS}, the best constant in the embedding is $C=2^{\frac1p-\frac1{\max\{p,r,s\}}}$. This proves the first equation of the lemma. 

Suppose the functions $g_0,g_1\in W_n(L^r,\ell^s)$ are disjointly supported and define $b=(b_{k,i})_{k\in\mathbb Z^n;i\in\mathbb Z_2}$ by setting $b_{k,i}^r=\int_{Q_k}|g_i|^r$. Disjointness shows that $b_{k,0}^r+b_{k,1}^r=\int_{Q_k}|g_0+g_1|^r$. For any constant $C$, the inequality 
\[
\|g_0+g_1\|_{W_n(L^r,\ell^s)}\le C\big(\|g_0\|_{W_n(L^r,\ell^s)}^q+\|g_1\|_{W_n(L^r,\ell^s)}^q\big)^{1/q},
\]
appropriately modified if $q=\infty$, holds for all such $g_1$ and $g_2$ if and only if the mixed norm embedding
\[
\|\|b\|_{\ell^r(\mathbb Z_2)}\|_{\ell^s(\mathbb Z^n)}\le C\|\|b\|_{\ell^s(\mathbb Z^n)}\|_{\ell^q(\mathbb Z_2)}
\]
holds for all non-negative $b$. By \cite[Theorem 4.1]{GS}, the best constant in the embedding is $C=2^{\frac1{\min\{q,r,s\}}-\frac1q}$ This proves the second equation of the lemma.
\end{proof}

With these estimates in hand we give a Christ-Kiselev maximal theorem for amalgams.
\begin{theorem}\label{maxamalg} Let $r,s,\tilde r,\tilde s\in(0,\infty]$, suppose $T$ is a bounded linear operator from $W_n(L^r,\ell^s)$ to $W_n(L^{\tilde r},\ell^{\tilde s})$ and let $T^*$ be the maximal operator based on $T$ associated to a filtration of Lebesgue measurable subsets of $\mathbb R^n$. If $\max\{r,s\}<\min\{\tilde r,\tilde s\}$, then $T^*$ is bounded from $W_n(L^r,\ell^s)$ to $W_n(L^{\tilde r},\ell^{\tilde s})$ and 
\[
\|T^*\|_{W_n(L^r,\ell^s)\to W_n(L^{s'},\ell^{r'})}\le \gamma
\|T\|_{W_n(L^r,\ell^s)\to W_n(L^{s'},\ell^{r'})},
\]
where $\gamma$ depends only on $r$, $s$, $\tilde r$, and $\tilde s$. If $\tilde r\ge1$ and $\tilde s\ge1$, then
\[
\gamma=\frac{2^{1/q}}{1-(2^{p/q}-1)^{1/p}}.
\]
where $p=\max\{r,s\}$ and $q=\min\{\tilde r,\tilde s\}$.
\end{theorem}
\begin{proof} By Lemma \ref{ulamalg}, $p=\max\{r,s\}$ implies $\ell_{(p),2}(W_n(L^r,\ell^s))=1$ and $q=\min\{\tilde r,\tilde s\}$ implies $u^{(q),2}(W_n(L^{s'},\ell^{r'}))=1$. It is easy to see that the triangle inequality in $W_n(L^{\tilde r},\ell^{\tilde s})$ holds with constant $\kappa=\max\{1,2^{1/\tilde r}\}\max\{1,2^{1/\tilde r}\}$. Theorem \ref{thmChK2} gives the conclusion.  
\end{proof}

\subsection{Application: A Menshov-Paley-Zygmund theorem}
Let $\mathcal I$ be the collection of compact subintervals of $\mathbb R$ and define the maximal Fourier transform, $\mathcal F_*$, of a locally integrable function $f$ by
\[
\mathcal F_*f(x)=\sup_{I\in\mathbb I}\Big|\int_Ie^{-2\pi ixt}f(t)\,dt\Big|\,.
\]
Note that the collection of compact intervals is not a filtration.

Next is a Menshov-Paley-Zygmund theorem for amalgams. We offer it as an application of the previous theorem, not as a new result. Indeed, pointwise convergence of the Fourier transform and boundedness of the maximal Fourier transform in Wiener amalgams have been well studied. See, for example, \cite{FW}. 
\begin{theorem}\label{MPZ}
Let $1\le r<2$ and $1\le s<2$. Then 
\[
\|\mathcal F_*\|_{W_1(L^r,\ell^s)\to W_1(L^{s'},\ell^{r'})}\le\frac{2^{1/p'}4}{1-(2^{p-1}-1)^{1/p}}\|\mathcal F\|_{W_1(L^r,\ell^s)\to W_1(L^{s'},\ell^{r'})},
\]
where $p=\max\{r,s\}$.
Consequently, if $f\in W_1(L^r,\ell^s)$, then for almost every $x\in \mathbb R$,
\[
\lim_{N_-,N_+\to\infty}\int_{-N_-}^{N_+}e^{-2\pi ixt}f(t)\,dt=f(x).
\]
\end{theorem}
\begin{proof}  First observe that for each locally integrable function $f$ on $\mathbb R$, the supremum that defines $\mathcal F_*f$ may be taken over the intervals with rational endpoints. This has no effect on the value of $\mathcal F_*f$.

Set $\mathcal A=\{[0,y]:0\le y\in\mathbb Q\}$ and let $\mathcal F_+^*$ and $\mathcal F_-^*$ be the maximal operators based on the one-dimensional Fourier transform and associated to the filtrations $\mathcal A$ and $-\mathcal A$, respectively.  Since every interval with rational endpoints is a union or difference of two elements of $\mathcal A\cup(-\mathcal A)$ we see that $\mathcal F_*\le2\mathcal F_+^*+2\mathcal F_-^*$. Since $\max\{r,s\}<2<\min\{r',s'\}$, the first conclusion of the theorem follows from Theorem \ref{maxamalg} and the Hausdorff-Young inequality for amalgams mentioned above. Since $L^1(\mathbb R)$ is dense in $W_1(L^r,\ell^s)$, the second conclusion follows from the first by a standard argument. See, for example, \cite[Corollary 8.10]{Tao}. 
\end{proof}

\section{For K\"othe dual operators}

Under mild conditions, a bounded, linear map, $T:E\to F$, between Banach function lattices has a K\"othe dual operator $T':F'\to E'$, where $E'$ and $F'$ denote the K\"othe dual spaces of $E$ and $F$, respectively. This is not true for sublinear operators such as maximal operators associated to filtrations. However, the upper and lower estimates that guarantee boundedness of the maximal operator in Theorem \ref{thmChK} imply lower and upper estimates on the dual spaces, enabling us to prove boundedness of the maximal K\"othe dual operator from conditions on the original spaces.

Let $E$ and $F$ be Banach function lattices over $\sigma$-finite measure spaces $(\Omega, \mu)$ and $(\widetilde\Omega, \nu)$, respectively, and let $T: E \to F$ be a bounded linear operator. A linear map $T': F' \to E'$ that satisfies
\[
\int_{\widetilde\Omega} gTf\,d\nu = \int_{\Omega} fT'g\, d\mu, \quad f\in E, g\in F',
\]
is called the K\"othe dual of $T$. In this case $T': F' \to E'$ is a bounded operator
with $\|T'\|_{F'\to E'} \leq \|T\|_{E\to F}$. If $E$ has order continuous norm, then every bounded, linear operator $T: E \to F$ has a K\"othe dual operator. (See \cite[Theorem 2.3]{MM}).

\begin{corollary} \label{dual} Let $E$ and $F$ be Banach function lattices with the weak Fatou property. Suppose $T': E'\to F'$ is the K\"othe dual of a bounded, linear operator $T:E\to F$ and let $(T')^{*}$ be the  maximal operator based on $T'$ associated to a filtration.
\begin{itemize}
\item[{\rm(i)}]\label{duala} If there exist $1\le p<q\le\infty$ such that $\ell_{(p), 2}(E)u^{(q), 2}(F)<2^{\frac1p-\frac1q}$, then $\|(T')^{*}\|_{F'\to E'}\le C\|T'\|_{F'\to E'}$ for some constant $C$ depending only on $p$, $q$, $E$, and $F$.

\item[{\rm(ii)}]\label{dualb} If there exist $1\le p<q\le\infty$ such that $E$ satisfies a lower $p$-estimate and $F$ satisfies an upper $q$-estimate,  then  
\[
\|(T')^{*}\|_{F'\to E'}\le \frac{2^{1/p'}}{1-(2^{q'/p'}-1)^{1/q'}}\ell_{(p)}(E)u^{(q)}(F)\|T'\|_{F'\to E'}
\]
\end{itemize}
\end{corollary} 
\begin{proof} We will apply Theorem \ref{thmChK2} to the operator $T':F'\to E'$. Since $E'$ is a K\"othe dual, it has the Fatou property. Suppose there exist $p,q$ such that $1\le p<q\le\infty$ and $\ell_{(p), 2}(E)u^{(q), 2}(F)<2^{\frac1p-\frac1q}$. Note that $\frac1p-\frac1q=\frac1{q'}-\frac1{p'}$. By Proposition \ref{ulndual} we have $1\le q'<p'\le\infty$ and $\ell_{(q'), 2}(F')u^{(p'), 2}(E')<2^{\frac1{q'}-\frac1{p'}}$. Now Theorem \ref{thmChK2} gives (i).

Now suppose there exist $p,q$ such that $1\le p<q\le\infty$, $E$ satisfies a lower $p$-estimate and $F$ satisfies an upper $q$-estimate. Proposition \ref{ulndual} shows that $F'$ satisfies a lower $q'$-estimate and $E'$ satisfies an upper $p'$-estimate. Theorem \ref{thmChK} gives (ii).
\end{proof}

\subsection{Application: The Fourier Transform in Lorentz $\Gamma$-spaces}
Under mild conditions, the Fourier transform on $\mathbb R^n$ with Lebesgue measure is its own K\"othe dual. So Proposition \ref{ulndual} can be used to prove boundedness of maximal operators based on the Fourier transform associated to a filtration. 

\begin{lemma} \label{dualF}
Let $E$ and $F$ be Banach function lattices over $\mathbb{R}^n$ equipped with Lebesgue measure. Suppose that $E$ and $F'$ are order continuous and there exists $C>0$ such
that $\|\widehat{f}\|_F \le C \|f\|_E$ for all $f\in E\cap L^1$. Then the Fourier transform extends from $E\cap L^1$ to a bounded linear operator from $E$ to $F$. It also extends from $F'\cap L^1$ to a bounded linear operator from $F'$ to $E'$. The second extension is the K\'othe dual of the first.
\end{lemma}
\begin{proof} Since Lebesgue measure on $\mathbb R^n$ is $\sigma$-finite, for each measurable  $f$ there exist non-negative simple functions $f_k\in L^1$ such that $f_k\uparrow |f|$ a.e. If $f\in E$, then each $f_k\in L^1\cap E$ and order continuity implies $\|f-f_k\operatorname{sgn}(f)\|_E\to0$. It follows that $E\cap L^1$ is dense in $E$. Since $F$ is complete, the Fourier transform extends to $E$ and $\|\widehat{f}\|_F \le C \|f\|_E$ for all $f\in E$. It follows that $\hat fg$ is integrable for all $g\in F'$.

For $g\in F'\cap L^1$ and $f\in E\cap L^1$ with $\|f\|_E\le 1$, we have $\|\hat f\|_F\le C$ so
\[
\int_{\mathbb R^n} f\hat g=\int_{\mathbb R^n}\hat f g\le\|\hat f\|_F\|g\|_{F'}\le C\|g\|_{F'}.
\]
Take the supremum over all such $f$ to get $\|\hat g\|_{E'}\le C\|g\|_{F'}$ for all $g\in F'\cap L^1$. As above we use the order continuity of $F'$ to see that $F'\cap L^1$ is dense in $F'$. Since $E'$ is complete, the Fourier transform extends to all of $F'$ and $\|\hat g\|_{E'}\le C\|g\|_{F'}$ for all $g\in F'$. It follows that $f\hat g$ is integrable for all $f\in E$. 

Let $f\in E$ and $g\in F'$. Choose $f_k\in E\cap L^1$ and $g_k\in F'\cap L^1$ such that $f_k\to f$ in $E$ and $g_k\to g$ in $F'$. Then $\hat f_k\to\hat f$ in $F$ and $\hat g_k\to\hat g$ in $E'$. We conclude that $f_k\hat g_k\to f\hat g$ in $L^1$ and $\hat f_kg_k\to \hat f g$ in $L^1$. Since $\int_{\mathbb R^n} f_k\hat g_k=\int_{\mathbb R^n}\hat f_k g_k$ for all $k$, we have $\int_{\mathbb R^n} f\hat g=\int_{\mathbb R^n}\hat f g$. This shows that the Fourier transform extended to $F'$ is the K\"othe dual of the Fourier transform extended to $E$.
\end{proof} 

Let $1< r<\infty$ and let $w$ be a weight. The Lorentz $\Gamma$-space $\Gamma_{p,w}$ is defined to be the set of all $f\in L^0_\mu $ such that
\[
\|f\|_{\Gamma_{r,w}}  = \bigg(\int_0^\infty (f^{**})^r w\bigg)^{1/r} <\infty.
\]
Here $f^{**}(t)=\frac1t\int_0^t f^*$. Unlike the situation for $\Lambda$-spaces, $\Gamma_{r,w}$ is a Banach space for any $w$. When $w$ is nonincreasing $\Lambda_{r,w}=\Gamma_{r,w}$, with equivalent norms. The property we need here is that the K\"othe dual of $\Lambda_{r,w}$ is a $\Gamma$-space: If $\int_0^\infty w=\infty$ and $\Lambda_{r,w}$ is normable, then $\Lambda_{r,w}=\Gamma_{r',\tilde w}$, with equivalent norms, where $\tilde w(x)=(W(x)/x)^{-p'}w(x)$. See \cite[Remark after Theorem 1]{Saw}.

\begin{corollary}\label{Gam}  Let $1<r<s<\infty$ with $s\ge2$ and let $v$ and $w$ be nonincreasing weights such that $\int_0^\infty v=\infty$, $\int_0^\infty w=\infty$, and
\[
\sup_{z>0}zV(1/z)^{1/r}W(z)^{-1/s}<\infty.
\]
If there exists a $p$, with $r<p<s$, such that $t\mapsto W(t^{p/r})$ is convex, then $\mathcal F^*$, the maximal operator based on the Fourier transform and associated to a filtration of Lebesgue measurable subsets of $\mathbb R^n$, is bounded from $\Gamma_{s',\tilde v}$ to $\Gamma_{r', \tilde w}$.
\end{corollary}
\begin{proof} Under these hypotheses, the proof of Corollary \ref{FTLor} shows that $\mathcal F$ is bounded from $\Lambda_{r,w}$ to $\Lambda_{s,v}$. It also shows that with $p$ chosen as above and $q=s$, the convexity and concavity conditions of Theorem \ref{Lorentz} are satisfied. But in the proof of Theorem \ref{Lorentz}, those conditions are shown to imply that $\Lambda_{r,w}$ has a lower $p$-estimate and $\Lambda_{s,v}$ has an upper $q$-estimate.

For the next step we need to know that the $\Lambda_{r,w}$ and $\Lambda_{s,v}$ have the weak Fatou property. In fact they have the Fatou property. This follows from the monotone convergence theorem and a well-known property of the rearrangement: If $f_n\uparrow f$ a.e., then $f_n^*\uparrow f^*$ a.e. Now Corollary \ref{dual}(ii) applies and we conclude that $\mathcal F'$ is bounded from $\Lambda_{s,v}'$ to $\Lambda_{r,w}'$, that is, from  $\Gamma_{s',\tilde v}$ to $\Gamma_{r', \tilde w}$.

It remains to show that $\mathcal F'$ is $\mathcal F$ itself. This follows from Lemma \ref{dualF}, once we verify that $\Lambda_{r,w}$ and  $\Gamma_{s',\tilde v}$ are order continuous. Among non-negative functions, $f\mapsto (f^*)^r$ and $f\mapsto (f^{**})^{s'}$ preserve order. So we readily reduce these assertions to the order continuity of $L^1$, which follows from Lebesgue's dominated convergence theorem.
\end{proof}

\section{Interpolation of Christ-Kiselev maximal operators}

In this section we briefly discuss how to combine our Christ-Kiselev maximal theorem on quasi-Banach function lattices with interpolation of sublinear operators to establish boundedness of maximal operators associated to filtrations.

\begin{theorem} \label{subint} Let $(E_0, E_1)$ and $(F_0, F_1)$ be couples of quasi-Banach function lattices over measure spaces $(\Omega,\mu)$ and $(\widetilde\Omega,\nu)$, respectively. Assume that $F_0$ and $F_1$ have the Fatou property. Let $\mathcal A$ be a filtration of $\Omega$, suppose $T:E_0+E_1\to F_0+F_1$ is a linear operator that is bounded from $E_i$ to $F_i$ for $i=0,1$, and let $T^*$ be the maximal operator based on $T$ associated to $\mathcal A$. 

If, for $i=0,1$, there exist $0<p_i<q_i<\infty$ such that $E_i$ satisfies a lower $p_i$-estimate and $F_i$ satisfies an upper $q_i$-estimate, then, whenever $E$ and $F$ are relative interpolation spaces with respect to $(E_0, E_1)$ and $(F_0, F_1)$, $T^*$ is bounded from $E$ to $F$ and 
\[
\|T^*\|_{E\to F}\le \delta\max_{i=0,1}\|T\|_{E_i\to F_i}.
\]
Here $\delta$ is a constant depending only on $p_i$, $q_i$, $E_i$, $F_i$, $E$ and $F$, for $i=0,1$.
\end{theorem}
\begin{proof} Suppose that $p_0$, $q_0$, $p_1$ and $q_1$ have the properties assumed above and let $E$ and $F$ be relative interpolation spaces with respect to $(E_0, E_1)$ and $(F_0, F_1)$. By \cite[Theorem 2.4.2]{BL} there is a constant $\eta$, depending only on the spaces $E_0$, $E_1$, $E$, $F_0$, $F_1$, and $F$, such that for all linear operators $T:E_0+E_1\to F_0+F_1$ that are bounded from $E_i$ to $F_i$ for $i=0,1$,
\[
\|T\|_{E\to F}\le\eta\max_{i=0,1}\|T\|_{E_i\to F_i}.
\]
It can be shown that if $S:E_0+E_1\to F_0+F_1$ is sublinear and is bounded from $E_i$ to $F_i$ for $i=0,1$, then
\[
\|S\|_{E\to F}\le\eta\max_{i=0,1}\|S\|_{E_i\to F_i}.
\]
This is a special case of a result from \cite{Bu}, but may also be proved directly using \cite[Theorem 1.25]{AB}, the Hahn-Banach extension theorem for operators taking values in a Dedekind complete Riesz space.

For $i=0,1$, Theorem \ref{thmChK} provides a $\delta_i$, depending only on $p_i$, $q_i$, $\ell_{(p_i)}(E_i)$, and $u^{(q_i)}(F_i)$, such that $\|T^*\|_{E_i\to F_i}\le\delta_i\|T\|_{E_i\to F_i}$. Therefore 
\[
\|T^*\|_{E\to F}\le\eta\max_{i=0,1}\|T^*\|_{E_i\to F_i}\le\delta\max_{i=0,1}\|T\|_{E_i\to F_i},
\]
where $\delta=\eta\max\{\delta_0,\delta_1\}$.
\end{proof}

Using classical Lorentz spaces as endpoints, we get the following.
\begin{corollary} \label{subLpq} Let $1\le p_i<q_i<\infty$ for $i=0,1$, and suppose 
\[
T: L_{p_0,1}(\mu)+L_{p_1, 1}(\mu) \to L_{q_0, \infty}(\nu)+ L_{q_1, \infty}(\nu)
\]
is bounded from $L_{p_i,1}(\mu)$ to $L_{q_i, \infty}(\nu)$ for $i=0,1$. Let $T^*$ be the maximal operator based on $T$ associated to a filtration.

If quasi-Banach function lattices $E$ and $F$ are relative interpolation spaces with respect to the couples $(L_{p_0,1}(\mu), L_{p_1, 1}(\mu))$ and $(L_{q_0, \infty}(\nu), L_{q_1, \infty}(\nu))$, then $T^*$ is bounded from $E$ to $F$. 

In particular, $T^{*}: L_{p,r}(\mu) \to L_{q, s}(\nu)$ is bounded whenever $\theta\in (0,1)$,
\[
p_0\ne p_1,\ \ \frac1p=\frac{1-\theta}{p_0}+\frac\theta{p_1},\ \  q_0\ne q_1,\ \  \frac1q=\frac{1-\theta}{q_0}+\frac\theta{q_1},\ \ \text{and}\ \ 0<r\le s\le\infty.
\]
\end{corollary}
\begin{proof} Let $q\in[1,\infty)$. Since $0\le f_n\uparrow f$ implies $\nu_{f_n}\uparrow \nu_f$, it is easy to see that $L_{q, \infty}(\nu)$ has the Fatou property. If $f_1,\dots, f_n$ have disjoint supports and sum to $f$, then $\nu_f=\sum_{j=1}^n\nu_{f_j}$ so
\[
\sup_{\lambda>0}\lambda\nu_f(\lambda)^{1/q}=\sup_{\lambda>0}\Big(\lambda^{q}\sum_{j=1}^nv_{f_j}(\lambda)\Big)^{1/q}\le\Big(\sum_{j=1}^n\sup_{\lambda>0}\lambda^{q}v_{f_j}(\lambda)\Big)^{1/q}.
\]
This shows $L_{q,\infty}(\nu)$ satisfies an upper $q$ estimate with constant $1$. 

In the proof of Theorem \ref{Lorentz}, we showed that if $r\le p$ and $W(t^{p/r})$ is convex, then $\Lambda_{r,w}$ has a lower $p$-estimate. Taking $r=1$ and $w(t)=\frac1pt^{\frac1p-1}$ we see that $W(t^{p/r})=t$ so $L_{p,1}(\mu)$ has a lower $p$-estimate when $1\le p$. 

Applying these observations to $q=q_i$ and $p=p_i$, for $i=0,1$, we verify the hypotheses of Theorem \ref{subint} and conclude that $T^*$ is bounded from $E$ to $F$. 

The second statement follows from the general Marcinkiewicz interpolation theorem (see \cite[Theorem 5.3.2]{BL}). It shows that $L_{p,r}(\mu)$ and $L_{q, s}(\nu)$ are relative interpolation spaces with respect to the couples $(L_{p_0,1}(\mu), L_{p_1, 1}(\mu))$ and $(L_{q_0, \infty}(\nu), L_{q_1, \infty}(\nu) )$.
\end{proof}

\subsection{Application: A convolution operator of Bak and Seeger}

Our next application is motivated by Bak and Seeger's paper \cite{BS}, in which the authors prove a remarkable endpoint version of the Stein-Thomas restriction theorem: Suppose that $0<a<n$, $0<b\leq a/2$, and $\mu$ is a Borel probability measure on $\mathbb R^n$ such that
\[
\sup_{r \in (0, 1],\, x\in \mathbb R^n} r^{-a}\mu(B(x, r))<\infty\quad\text{and}\quad\sup_{|\xi|\geq 1} |\xi|^b |\widehat{\mu}(\xi)|<\infty,
\]
where $B(x, r)$ is the Euclidean ball with center $x$ and radius $r$. Then the Fourier transform is bounded from the Lorentz space $L_{p_0, 2}$ over $\mathbb{R}^n$ to $L_2(\mu)$ with $p_0  = \frac{2(n-a +b)}{2(n-a) + b}$.

The proof of their result is a consequence of a Lorentz norm inequality for a certain convolution operator, see \cite[Proposition 2.1]{BS}. We show that the inequality also holds for a maximal operator based on the convolution operator and associated to a filtration.
\begin{proposition}\label{BaSe} With $a$, $b$, $n$ and $\mu$ as above, let 
\[
\rho = \frac{(n-a + 2b)(n-a + b)}{(n -a)^2 + 3b(n-a) + b^2} \quad\text{and}\quad\sigma = \frac{n-a + 2b}{b}.
\]
If $p\in (\rho, \sigma')$ and $\frac1p-\frac1q = \frac{n-a}{n-a + b}$, then for any $s\in (0, \infty]$
the maximal operator $T^{*}$, based on the convolution operator $Tf  = f\ast \widehat{\mu}$ and associated to a filtration of $\mathbb R^n$, is bounded from $L_{p,s}$  to $L_{q, s}$. 
\end{proposition}

\begin{proof} From \cite[Proposition 2.1]{BS}, $T$ is bounded from  $L_{\rho, 1}$ to $L_{\sigma, \infty}$ and from $L_{\sigma', 1}$ to $L_{\rho', \infty}$.

Now observe that $1<\rho <\sigma<\infty$ and so $1<\sigma' < \rho' <\infty$. A calculation shows that $\rho <\sigma'$, so for any $p\in (\rho, \sigma')$ we can find $\theta \in (0, 1)$ such that $1/p = (1-\theta)/\rho + \theta/\sigma'$.
Since 
\[
\frac1\rho - \frac1\sigma = \frac{n-a}{n-a + b},
\]
it follows that
\[
\frac1q = \frac1p - \frac{n-a}{n-a + b} = \frac{1-\theta}{\sigma} + \frac{\theta}{\rho'}.
\]
Consequently, the result follows from Corollary \ref{subLpq}.
\end{proof}

%\bibliographystyle{plain}
%\bibliography{CKREFS}

\end{document}